\documentclass{amsart}
\usepackage{amssymb}   
\usepackage{amsmath}
\usepackage{amsthm}
\usepackage[all,dvips]{xy}
\usepackage[]{pstricks,pst-node,pst-plot}
\usepackage{graphicx}
\usepackage{pstcol,pst-char}

\def\ZZ{{\mathbb Z}}
\def\Acal{{\mathcal A}}
\def\Ccal{{\mathcal C}}
\def\Ocal{{\mathcal O}}
\def\H{\rm{H}}

\newtheorem{theorem}{Theorem}[section]
\newtheorem{lemma}[theorem]{Lemma}
\newtheorem{proposition}[theorem]{Proposition}
\newtheorem{corollary}[theorem]{Corollary}

\newtheorem{definition-lemma}[theorem]{Definition-Lemma}

\theoremstyle{definition}
\newtheorem{definition}[theorem]{\bf Definition}

\theoremstyle{remark}
\newtheorem{remark}[theorem]{\bf Remark}

\def\vandaag{\number\day\space\ifcase\month\or
 januari\or februari\or  maart\or  april\or mei\or juni\or  juli\or
 augustus\or  september\or  oktober\or november\or  december\or\fi,
\number\year}
\def\today{\ifcase\month\or
 Jan\or Febr\or  Mar\or  Apr\or May\or Jun\or  Jul\or
 Aug\or  Sep\or  Oct\or Nov\or  Dec\or\fi
 \space\number\day, \number\year}

\begin{document}

\title[Invariants in positive characteristic]{Relations between some invariants of algebraic varieties in positive characteristic}

\author{Gerard van der Geer $\,$}
\address{Korteweg-de Vries Instituut, Universiteit van
Amsterdam, Postbus 94248, 1090 GE Amsterdam,  The Netherlands}
\email{G.B.M.vanderGeer@uva.nl}
\author{Toshiyuki Katsura}
\address{Department of Mathematics, Hosei University, 184-8584 Tokyo}
\email{toshiyuki.katsura.tk@hosei.ac.jp}
\subjclass{14G, 14G17, 11G25}
\begin{abstract} We discuss relations between certain invariants of varieties
in positive characteristic, like the $a$-number and the height of the
Artin-Mazur formal group. We calculate the $a$-number for Fermat surfaces.
\end{abstract}
\maketitle
\begin{section}{Introduction}\label{sec:intro}
Algebraic varieties in positive characteristic possess special invariants that
have no analogue in characteristic $0$. In this paper we consider
three such invariants related to the cohomology groups ${\H}^n$ with $n$
equal to the dimension of the variety. 
The first, the $a$-number that we introduced in \cite{GK3},
registers where the image of Frobenius acting on ${\H}^n(X,{\Ocal}_X)$
lands in the Hodge filtration of ${\H}^n_{\rm dR}(X)$. 
For abelian varieties it coincides with the
$a$-number defined by Oort. The second one is the height $h(X)$ of the
Artin-Mazur formal group, which is an infinitesimal invariant 
related to the \'etale cohomology group ${\H}^n_{\rm et}(X,{\bf G}_m)$.
The third, baptised the $b$-number $b(X)$, 
is related to ${\H}^n(X,B_i)$ with $B_i$ the sheaves of exact $1$-forms
defined by Illusie.

These invariants are related in subtle ways. In this note we prove the relation
$$
h(X)= b(X)+p_g(X)
$$
and in case $b(X)<\infty$ the basic estimate 
$$
b(X) \leq \dim {\H}^{n-1}(X,\Omega_X^1)\, .
$$
Furthermore we prove that if $b(X)<\infty$ then $a(X)=0$ or 
$a(X)=1$ and  we show that $a=0$ if and only if $b=0$.
As an example we calculate the $a$-number for Fermat surfaces.

Throughout this paper $k$ denotes an algebraically closed field of
characteristic $p>0$ and all varieties considered are complete and
non-singular.

\end{section}
\begin{section}{The $a$-number}
Let $X$ be a complete non-singular variety of dimension $n$ over $k$. 
The de Rham cohomology of $X$ is the hypercohomology of the complex 
$(\Omega_X^{\bullet}, d)$. We are interested in the $n$-th de Rham
cohomology group and its Hodge filtration
$$
{\H}^n_{\rm dR}(X)={\rm F}^0 \supset {\rm F}^1 \supset \cdots 
\supset {\rm F}^n \supset (0) \, .
$$
We shall assume in this section that the Hodge-to-de Rham spectral
sequence
$$
E^{i,j}_1= {\H}^j(X,\Omega_X^i) \Longrightarrow {\H}^{i+j}_{\rm dR}(X)
$$
degenerates at the $E_1$-level. 
(This condition is for example 
fulfilled if the characteristic $p$ satisfies $p>n$ and
if $X$ can be lifted to the Witt ring $W_2(k)$ of length $2$, 
cf.\ Deligne-Illusie \cite{DI}.) In this case 
the graded pieces are
$$
{\H}^n(X,{\Ocal}_X)\cong {\rm F}^0/{\rm F}^1, \, 
{\H}^{n-1}(X,\Omega^1_X)\cong {\rm F}^1/{\rm F}^2, \ldots, 
{\H}^0(X,\Omega^n_X)\cong {\rm F}^n\, .
$$
We have a Frobenius morphism $F$ acting on ${\H}^n_{\rm dR}(X)$; it acts
by $0$ on ${\rm F}^1$ and it induces a homomorphism
$$
F: {\H}^n(X,{\Ocal}_X)\cong {\rm F}^0/{\rm F}^1 
\longrightarrow {\H}^n_{\rm dR}(X)\, .
$$
By Katz \cite{Kz} this induced homomorphism is injective. 

In \cite{GK3} we defined an invariant called the $a$-number as follows.
\begin{definition}
The $a$-number of the variety $X$ is defined by 
$$
a(X)= \max \{ i: ({\rm Im}\, F)\cap
{\rm F}^i \neq (0) \}\, .
$$
\end{definition}
\noindent
Note that $0 \leq a(X) \leq n=\dim (X)$.
\smallskip

For an abelian variety $X$ Oort defined in \cite{O1} the $a$-number $a(X)$ by 
$$
a(X)= \dim_k{\rm Hom}(\alpha_p,X),
$$
with $\alpha_p$ the local-local group scheme (that is, the kernel of
Frobenius acting on the additive group ${\bf G}_a$). Since ${\rm End}(\alpha_p)
\cong k$ we can view ${\rm Hom}(\alpha_p,X)$ as a right vector space over $k$.
The union of all the images of the group scheme homomorphisms $\alpha_p
\to X$ is the maximal subgroup scheme $A(X)$ of $X[p]$ annihilated by
the operators $F$ (Frobenius) and $V$ (Verschiebung) on the kernel $X[p]$ of
multiplication by $p$. Then this $a$-number is $\log_p {\rm ord}\, A(X)$.
The Dieudonn\'e module of $X[p]$ can be identified with ${\H}^1_{\rm dR}(X)$
and the Dieudonn\'e module of $A(X)$ can be identified with the
kernel of $V$ acting on ${\H}^0(X,\Omega^1_X)$, the kernel of $F$.
Since $\ker V= {\rm Im}\, F$ Oort's $a$-number equals the dimension 
of the image of $F$ (acting on ${\H}^1_{\rm dR}(X)$) 
in ${\H}^0(X,\Omega_X^1)$:
$$
\dim_k{\rm Hom}(\alpha_p,X)= \dim {\H}^0(X,\Omega_X^1) \cap ({\rm Im}\, F).
$$
We showed in \cite{GK3} that for an abelian variety our definition of
the $a$-number that involves ${\H}^n_{\rm dR}(X)$ 
coincides with Oort's definition that involves ${\H}^1_{\rm dR}(X)$. 
For the reader's convenience we recall the proof.

\begin{proposition}
For an abelian variety $X/k$ the two definitions of the $a$-number coincide.
\end{proposition}
\begin{proof}
Recall that ${\H}_{\rm dR}^n(X) = \wedge^n {\H}^1_{\rm dR}(X)$ and
if we write ${\H}^1_{\rm dR}(X)=V_1\oplus V_2$ with $V_1={\H}^0(X, \Omega_X^1)$
and $V_2$ a complementary subspace, then
the Hodge filtration on ${\H}^n_{\rm dR}(X)$ is
${\rm F}^r= \sum_{j=r}^n \wedge^j V_1 \otimes \wedge^{n-j} V_2$.
We have  $F({\H}^n(X,{\Ocal}_X))= F(\wedge^n {\H}^1(X,{\Ocal}_X))=
\wedge^n F({\H}^1(X,{\Ocal}_X))$. If we write $F({\H}^1(X,{\Ocal}_X))= 
A\oplus B$
with $A$ the intersection with ${\H}^0(X,\Omega_X^1)$ 
and $B$ a complementary space, then
$\wedge^n (A\oplus B)= \wedge^a A \otimes \wedge^{n-a} B$ with $a=\dim (A)$.
It follows that $F({\H}^n(X,{\Ocal}_X))$ lies in
${\rm F}^a$, but not in ${\rm F}^{a+1}$.
\end{proof}

The $a$-number is just one of the invariants that can be associated to
the relative position of the Hodge filtration and the conjugate
filtration on ${\H}^n_{\rm dR}(X)$, cf.\ \cite{GK3}, \cite{MW}.
\end{section}
\begin{section}{The $h$-number}
For a non-singular proper variety $X/k$ of dimension $n$ 
one can consider the formal
completion of the Picard group. For any local artinian scheme $S$ 
with residue field $k$ its $S$-valued points are given by
the exact sequence
$$
0 \to \widehat{\rm Pic}(X)(S) \to {\H}^1_{\rm et}(X\times S, {\bf G}_m) \to
{\H}^1_{\rm et}(X,{\bf G}_m),
$$
where ${\bf G}_m$ denotes the multiplicative group. This invariant provides
interesting information, for example for elliptic curves in positive characteristic. This idea was generalized by
Artin and Mazur to the higher cohomology groups in \cite{AM}. 
Let ${\Ccal}$ the category of Artinian local algebras $(R,m)$ over $k$
with maximal ideal $m$  such that $R/m \cong k$, and denote by
 ${\Acal}$ the category of abelian groups. We consider
the covariant functor $\Phi_{X}=\Phi_X^{(n)} : {\Ccal} \longrightarrow {\Acal}$
defined by
$$
\Phi_{X}(R) = \ker({\H}_{\rm et}^{n}(X\times_{k}{\rm Spec}~R,{\bf G}_{m})
\longrightarrow {\H}_{\rm et}^{n}(X, {\bf G}_{m}))
$$
for $R$ an object in ${\Ccal}$. Here ${\H}_{et}^{n}(X, {\bf G}_{m})$ denotes
the $n$-th \'etale cohomology group with values in the multiplicative group
${\bf G}_{m}$, and $X \longrightarrow X\times_{k}{\rm Spec}~{R}$
is the natural immersion. When the functor $\Phi_X$ is pro-representable
by a formal Lie group,
we call the formal Lie group an Artin-Mazur formal group.
We also denote by $\Phi_X$ the formal Lie group.
The tangent space of $\Phi_X$ is given by ${\H}^{n}(X, {\Ocal}_{X})$
(cf.\ Artin-Mazur ~\cite{AM}). If $\Phi_X$ is pro-representable
by a formal Lie group we denote by $h(X)$
the {\sl height} of the corresponding formal Lie group $\Phi_X$ 
and call $h(X)$ the {\sl $h$-number} of $X$.

This $h$-number is a special invariant in positive characteristic and
$h(X)$ is either a positive integer or $\infty$. For example, in the case of
an elliptic curve the $h$-number assumes the values $1$ or $2$ depending
on whether the elliptic curve is ordinary or supersingular;
for a K3 surface $X$ we know that either $1\leq h(X)\leq 10$ or $h(X)=\infty$,
the latter if $\Phi_X$ is the additive group ${\bf G}_a$,
cf. Artin~\cite{A,AM}.

\end{section}
\begin{section}{The $b$-number}\label{sec:b-number}
Let $X$ be a complete non-singular variety over $k$. 
Following Illusie (\cite{I}) we define $B_1=d{\Ocal}_X$ and 
$Z^1=\ker (\Omega^1_X {\buildrel d \over \longrightarrow} \Omega^2_X)$. 
Using the Cartier operator 
$$
C: Z^1 \to \Omega_X^1
$$
we can define inductively sheaves of $O_X$-modules for $j\geq 2$ by
$$
\begin{aligned}
Z^j=& C^{-1}(Z^{j-1}), \cr
B_j=& C^{-1}(B_{j-1}).
\end{aligned}
$$
Since $B_1 \subseteq Z^1$ we get a filtration
$$
0 = B_{0} \subset B_{1}\subset B_{2} \subset \cdots \subset B_{i}\subset \cdots
\subset Z^{i}\subset \cdots \subset Z^{1} \subset Z^{0} = \Omega_{X}^{1} \, .
$$
Recall that the sheaf $B_i$ admits a description in terms of 
Witt vector cohomology by the Serre differential 
$D_i: W_i(\Ocal_X) \to B_i$ given by
$$
(a_0,\ldots,a_{i-1}) \mapsto  a_{0}^{p^{i-1}-1}da_{0} +a_{1}^{p^{i-2}-1}da_{1} + \cdots + 
a_{i-2}^{p-1}da_{i-2} + da_{i-1}.
$$
This map $D_i$ induces an isomorphism $$
W_{i}({\Ocal}_{X})/FW_{i}({\Ocal}_{X})\cong B_{i},
$$
where $F$ is the Frobenius operator, cf.\ Serre \cite{S}.
The cohomology of these sheaves $B_i$ and $Z^i$ leads to interesting invariants.
One example is:
\begin{definition}
The $b$-number of $X$ is $b(X)= \max_{i\geq 1} \dim {\H}^n(X,B_i)$\, .
\end{definition}
By the exact sequence
$$
       0 \rightarrow  B_{1} \longrightarrow B_{i+1}
           {\buildrel C \over \longrightarrow} B_{i}\rightarrow 0
$$
we have a surjective homomorphism
$$
        C : {\H}^{n}(X, B_{i+1})\longrightarrow {\H}^{n}(X, B_{i}),
$$
which gives a projective system $\{C, {\H}^{n}(X, B_{i})\}$
and we may reformulate the definition of $b(X)$ as
$$
b(X) =  \dim \lim_{\leftarrow} {\H}^{n}(X, B_{i}).
$$

The $b$-number can be related to the action of Frobenius on Witt vector
cohomology as follows.
\begin{proposition}\label{b-Witt}
We have  $b(X) = \dim_k {\H}^{n}(X, W({\Ocal}_{X}))/F{\H}^{n}(X, W({\Ocal}_{X}))$.
\end{proposition}
\noindent
Note that ${\H}^{n}(X, W({\Ocal}_{X}))/F{\H}^{n}(X, W({\Ocal}_{X}))$ is
a vector space over $W(k)/pW(k)\cong k$.
\begin{proof}
We have the commutative diagram
\begin{displaymath}
\begin{xy}
\xymatrix{
W_{i+1} \ar[r]^{D_{i+1}} \ar[d]^{R} & B_{i+1} \ar[d]^C \\
W_i \ar[r]^{D_i} & B_i \\
}
\end{xy}
\end{displaymath}
and a map of exact sequences
\begin{displaymath}
\begin{xy}
\xymatrix{
 &{\H}^n(X,W_{i+1}(\Ocal_X)) \ar[r]^F \ar[d]^R & 
{\H}^n(X,W_{i+1}(\Ocal_X)) \ar[r]  \ar[d]^R & {\H}^n(X,B_{i+1}) \ar[r] 
\ar[d]^C & 0 \\
 &{\H}^n(X,W_{i}(\Ocal_X)) \ar[r]^F &{\H}^n(X,W_{i}(\Ocal_X)) \ar[r] &
{\H}^n(X,B_{i}) \ar[r] & 0 \\
}
\end{xy}
\end{displaymath}
The projective system $\{R:{\H}^{n}(X, W_{i+1}({\Ocal}_{X}))\rightarrow 
{\H}^{n}(X, W_{i}({\Ocal}_{X}))\}$
satisfies the Mittag-Leffler condition, and by taking the projective limit 
we obtain an exact sequence
$$
{\H}^{n}(X, W({\Ocal}_{X})) 
\stackrel{F}{\longrightarrow} {\H}^{n}(X, W({\Ocal}_{X})) 
 \longrightarrow   \lim_{\leftarrow} {\H}^{n}(X, B_{i})  \rightarrow 0,
$$
and this gives us the desired conclusion.
\end{proof}
Another characterization of the $b$-number uses the exact sequence
$$
0 \to W_i(\Ocal_X) {\buildrel F \over \longrightarrow} W_i(\Ocal_X) 
{\buildrel D_i \over \longrightarrow} B_i \to 0
$$
and the induced long exact cohomology sequence.
\begin{proposition}\label{b-alternative}
We have 
$b(X)= \max_{i\geq 1} \dim {\H}^{n-1}(X,B_i)/{\rm Im}\, D_i $.
\begin{proof}
From the long exact cohomology sequence we get the exact sequence
$$
\begin{aligned}
0 \to {\H}^{n-1}(X,B_i)/{\rm Im}\,  D_i \to {\H}^n(X,W_i(\Ocal_X)) 
{\buildrel F \over \longrightarrow} {\H}^n(X,W_i(\Ocal_X)) & \\
{\buildrel D_i \over \longrightarrow} {\H}^n(X,B_i) \to 0 & . \\
\end{aligned}
$$
Looking at the lenghts of these $W_{i}(k)$-modules,
we observe
$$
\dim {\H}^{n -1}(X, B_{i}) /{\rm Im}\,  D_{i} = \dim {\H}^{n}(X, B_{i}).
$$
\end{proof}
\end{proposition}
\end{section}
\begin{section}{An inequality for the $b$-number}
We now will prove a basic inequality for the $b$-number. 

\begin{theorem}\label{thm:b-estimate}
If $b(X)< \infty$ then $b(X) \leq \dim {\H}^{n-1}(X, \Omega^1_X)$.
\end{theorem}

The natural inclusion $B_{\ell} \to \Omega^1_X$ induces a homomorphism
$$
\varphi_{\ell}: {\H}^{n-1}(X,B_{\ell})\longrightarrow {\H}^{n-1}(X,\Omega^1_X).
$$
The proof of this theorem relies on the following lemma relating the
kernel of $\varphi_{\ell}$ and the image of $D_{\ell}$.

\begin{lemma}\label{b-lemma}
If $b(X) < \infty$ then $\ker \varphi_{\ell} \subset {\rm Im}\, D_{\ell}$
for any positive integer~$\ell$.
\end{lemma}
\begin{proof} Assuming that the result does not hold we consider the smallest
positive $\ell$ such that $\ker \varphi_{\ell} \not\subset {\rm Im}\, D_{\ell}$.
Then there exists a non-zero element $\alpha \in {\H}^{n-1}(X, B_{\ell})$ 
such that $\varphi_{\ell}(\alpha) = 0$ and
$\alpha \notin {\rm Im}\, D_{\ell}$. Let $m$ be the non-negative integer 
such that $C^{m}(\alpha) \notin {\rm Im}\, D_{\ell - m}$
and $C^{m+1}(\alpha) \in {\rm Im}\, D_{\ell - m - 1}$. Here we define 
${\rm Im}\, D_{k}  = 0$ for $k \leq 0$.
For any positive integer $s$ the commutativity of the diagram
\begin{displaymath}
\begin{xy}
\xymatrix{
{\H}^{n-1}(X,W_{i+1}) \ar[r]^{D_{i+1}} \ar[d]^R & {\H}^{n-1}(X,B_{i+1}) \ar[d]^C \\
{\H}^{n-1}(X,W_i) \ar[r]^{D_i} & {\H}^{n-1}(X,B_i) \\
}
\end{xy}
\end{displaymath}
implies $C^{m+s}(\alpha) \in {\rm Im}\, D_{\ell - m - s}$

We take an affine open covering $\{U_{i}\}$ of $X$. 
Then $\alpha$ is given by a \v{C}ech cocycle $\{\alpha_I\}$ with 
$I=i_0i_1\cdots i_{n-1}$ and 
$$
\alpha_I= \sum_{j=0}^{\ell-1} (f_I^{(j)})^{p^{\ell-1-j}-1} df_I^{(j)}
$$
for $f_I^{(j)}=f_{i_0i_1\cdots i_{n-1}}^{(j)}\in \Gamma(U_{i_{0}}\cap U_{i_{1}}\cap \ldots \cap U_{i_{n-1}}, {\Ocal}_{X})$. 
Since by assumption $\varphi_{\ell}(\alpha) = 0$, there exist
elements $\omega_{i_{0}i_{1}\ldots i_{n-2}}\in
\Gamma (U_{i_{0}}\cap U_{i_{1}}\cap \ldots \cap U_{i_{n-2}}, \Omega_{X}^{1})$
such that
$$
\alpha_I= \omega_{i_{1}i_{2}\ldots i_{n-1}}-\omega_{i_{0}i_{2}\ldots i_{n-1}}
+ \cdots +
(-1)^{n-1}\omega_{i_{0}i_{1}\ldots i_{n-2}} \eqno(1)
$$
For an affine open set $U$, we have $H^{1}(U, B_{1}) = 0$. Thus the exact
sequence 
$$
0\to B_1 \to Z^1 {\buildrel C \over \longrightarrow} \Omega^1_X \to 0
$$ 
implies that the Cartier operator
$C: \Gamma(U, Z^{1}) \rightarrow \Gamma(U, \Omega^{1}_{X})$ is surjective.
So we can find an element $\tilde{\omega}_{i_{0}i_{1}\ldots i_{n-2}}
\in \Gamma(U_{i_{0}}\cap U_{i_{1}}\cap \ldots \cap U_{i_{n-2}}, \Omega_{X}^{1})$
that maps to ${\omega}_{i_{0}i_{1}\ldots i_{n-2}}$ under $C$
and then 
$$
\tilde{\omega}_{i_{1}i_{2}\ldots i_{n-1}}-\tilde{\omega}_{i_{0}i_{2}
\ldots i_{n-1}} + \cdots +(-1)^{n-1}
\tilde{\omega}_{i_{0}i_{1}\ldots i_{n-2}}
$$
maps under $C$ to the right hand side of (1).
Since $\omega \in \Gamma(U, Z^{1})$ has $C(\omega) = 0$
if and only if $\omega = df$ for a suitable regular function
$f \in \Gamma(U, {\Ocal}_{X})$,
we can choose a regular function 
$f_I^{(\ell)}=f_{i_{0}i_{1}\ldots i_{n-1}}^{(\ell)} \in 
\Gamma( U_{i_{0}}\cap U_{i_{1}}\cap \ldots \cap U_{i_{n-1}}, {\Ocal}_{X})$
such that
$$
\sum_{j=0}^{\ell} (f_I^{(j)})^{p^{\ell-j}-1} df_I^{(j)} = 
\tilde{\omega}_{i_{1}i_{2}\ldots i_{n-1}}-\tilde{\omega}_{i_{0}i_{2}
\ldots i_{n-1}} + \cdots +(-1)^{n-1}
\tilde{\omega}_{i_{0}i_{1}\ldots i_{n-2}}
$$
In this way the cochain $
C^{-1}(\alpha) =\sum_{j=0}^{\ell} (f_I^{(j)})^{p^{\ell-j}-1} df_I^{(j)}$
becomes a co-cycle in the \v{C}ech co-chains of the sheaf $B_{\ell + 1}$ and
gives an element of ${\H}^{n-1}(X, B_{\ell + 1})$ .
Repeating this procedure $t$ times we obtain $t$ elements
$$
\alpha, C^{-1}(\alpha), C^{-2}(\alpha), \ldots, C^{-t}(\alpha)
$$
in ${\H}^{n-1}(X, B_{\ell + t})$.

Consider the vector space ${\H}^{n-1}(X, B_{\ell + t})/{\rm Im}\, D_{\ell + t}$
over $k$. Note that the Cartier operator 
$C:  B_{\ell + i + 1}\rightarrow B_{\ell + i }$ induces a $p^{-1}$-linear mapping 
$$
C : {\H}^{n-1}(X, B_{\ell + i + 1})/{\rm Im}\, D_{\ell + i + 1} \longrightarrow
{\H}^{n-1}(X, B_{\ell + i })/{\rm Im}\,  D_{\ell + i}.
$$
Suppose the elements 
$\alpha, C^{-1}(\alpha), C^{-2}(\alpha), \ldots, C^{-t}(\alpha)$
are linearly dependent over $k$ in ${\H}^{n-1}(X, B_{\ell + t})/{\rm Im} \,
D_{\ell + t}$.
So there exist elements $a_{i}\in k$ with $i = 0, 1, \ldots, t$ 
such that
$$
a_{0}\alpha + a_{1}C^{-1}(\alpha) + a_{2}C^{-2}(\alpha)+ \ldots + a_{t} C^{-t}(\alpha)
= 0
$$
in ${\H}^{n-1}(X, B_{\ell + t})/{\rm Im} D_{\ell + t}$.
By letting $C^{t+m}$ operate on both sides we have
$$
      a_{t}^{p^{-t-m}} C^{m}(\alpha) =  0
$$
in ${\H}^{n-1}(X, B_{\ell- m}))/{\rm Im}\, D_{\ell - m}$.
By our assumption, $C^{m}(\alpha)$ is not contained in ${\rm Im}\, D_{\ell - m}$. 
It follows that $a_{t} = 0$.
Repeating this procedure we see $a_{0} = a_{1} = \ldots = a_{t} = 0$.
This means that our elements are linearly independent over $k$
and we see that 
$\dim {\H}^{n-1}(X, B_{\ell + t})/{\rm Im}\, D_{\ell + t} \geq t + 1$
for any positive integer $t$, which
contradicts the finiteness of $b(X)$. We conclude that $\ker \varphi_{\ell} 
\subset {\rm Im}\, D_{\ell}$
for any positive integer~$\ell$.
\end{proof}
We now prove Theorem \ref{thm:b-estimate}. 
We have by Proposition \ref{b-alternative} and its proof  
and by Lemma \ref{b-lemma} that
$$
\dim {\H}^n(X,B_i)= \dim {\H}^{n-1}(X,B_i)/{\rm Im}\, D_i 
\leq \dim {\H}^{n-1}(X, B_{i})/
\ker \varphi_{i}
$$
and since 
$\dim {\H}^{n-1}(X, B_{i})/\ker \varphi_{i} \leq 
\dim {\H}^{n-1}(X, \Omega_{X}^{1})$, we derive the inequality
$b(X) \leq \dim \H^{n-1}(X, \Omega_{X}^{1})$.
This concludes the proof of Theorem \ref{thm:b-estimate}.
\end{section}
\bigskip

We conclude this section with a remark about the 
spaces ${\H}^{n-1}(X,B_i)$.
Using the natural inclusion $\psi_{i} : B_{i} \rightarrow B_{i + 1}$
 we have the induced linear mapping
$$
\psi_{i} : {\H}^{n-1}(X, B_{i}) \longrightarrow
{\H}^{n-1}(X,B_{i+1}).
$$
\begin{corollary}
Assume $b(X) < \infty$. If for all $i$ the map 
$D_{i} : {\H}^{n-1}(X, W_{i}({\Ocal}_{X})) \rightarrow 
{\H}^{n-1}(X, B_{i})$ is zero, then
$\psi_{i} : {\H}^{n-1}(X, B_{i}) \longrightarrow
{\H}^{n-1}(X,B_{i+1})$
is injective for any $i \geq 1$.
In particular, if ${\H}^{n-1}(X, {\Ocal}_{X}) = 0$ then $\psi_{i}$ is
injective
for any $i \geq 1$.
\end{corollary}
\begin{proof}
The first part  of this corollary follows from the fact that
the composition of the homomorphisms
$$
{\H}^{n-1}(X,B_{i}) {\buildrel \psi_{i} \over \longrightarrow} 
{\H}^{n-1}(X, B_{i + 1}) {\buildrel \varphi_{i + 1} \over \longrightarrow} 
{\H}^{n-1}(X, \Omega_{X}^{1}),
$$
where $\varphi_{i} = \varphi_{i + 1}\circ \psi_{i}$, is injective.
We have an  exact sequence
$$
0 \rightarrow W_{i-1}({\Ocal}_{X})  {\buildrel V \over \longrightarrow} 
W_{i}({\Ocal}_{X}) {\buildrel R^{i-1} \over \longrightarrow}
{\Ocal}_{X} \rightarrow  0.
$$
So with ${\H}^{n-1}(X, {\Ocal}_{X}) = 0$ we find inductively
${\H}^{n-1}(X, W_{i}({\Ocal}_{X})) = 0$. 
Thus the second statement follows from the first one.
\end{proof}
\begin{section}{The relation between the $b$-number and the $h$-number}
Assuming in this section that the Artin-Mazur formal group 
$\Phi_X=\Phi^{(n)}_X$is pro-representable by a formal Lie group 
we establish a relation
between the $h$-number (=height) and the $b$-number.

\begin{theorem}
Let $X$ be a non-singular complete algebraic variety with the Artin-Mazur 
formal group $\Phi_{X}$ pro-representable by a formal Lie group and let
$h(X)$ be the height of $\Phi_{X}$. Then we have the equality
$$
h(X)=b(X)+p_g(X),
$$
with $p_g(X)=\dim {\H}^n(X,O_X)$ the geometric genus.
\end{theorem}
\begin{proof}
This follows on the one hand 
from our interpretation of the $b$-number in terms of
Witt vector cohomology (Lemma \ref{b-Witt}) and on the other hand by 
Dieudonn\'e theory that expresses the
height in terms of Witt vector cohomology as follows.
We use the covariant Dieudonn\'e module theory (Cartier Dieudonn\'e module 
theory). For our Artin-Mazur formal group, the Dieudonn\'e module is 
given by ${\H}^{n}(X, W({\Ocal}_{X}))$.  So
the general theory of Dieudonn\'e modules implies
$$
h(X) = \dim {\H}^{n}(X, W({\Ocal}_{X}))/p{\H}^{n}(X, W({\Ocal}_{X})).
$$
Since $p=VF$ we have an exact sequence
$$
\begin{aligned}
0 \to V {\H}^n(X,W({\Ocal}_X))/p{\H}^{n}(X, W({\Ocal}_{X})) \longrightarrow
{\H}^{n}(X, W({\Ocal}_{X}))/p{\H}^{n}(X, W({\Ocal}_{X})) &\\
\longrightarrow {\H}^{n}(X, W({\Ocal}_{X}))/V{\H}^{n}(X, W({\Ocal}_{X})) \to 0 & \\
\end{aligned}
$$
and we now have to calculate the dimensions of the second and 
fourth term in this sequence. By general Dieudonn\'e module theory
the Verschiebung $V$ acting on ${\H}^{n}(X, W({\Ocal}_{X}))$ is injective
so that $p=VF$ implies
$$
{\H}^{n}(X, W({\Ocal}_{X}))/F{\H}^{n}(X, W({\Ocal}_{X})) \cong
V{\H}^{n}(X, W({\Ocal}_{X}))/p{\H}^{n}(X, W({\Ocal}_{X}))
$$
and we know by Proposition 
\ref{b-Witt} that its dimension is $b(X)$. As to the dimension
of the fourth term we observe that the exact sequence
$$
0 \to V {\H}^{n}(X, W({\Ocal}_{X}))\longrightarrow 
{\H}^{n}(X, W({\Ocal}_{X})) \longrightarrow
{\H}^{n}(X, {\Ocal}_{X}) \to 0
$$
shows that the dimension of the fourth term is 
$\dim {\H}^{n}(X, {\Ocal}_{X}) =p_g(X)$.
\end{proof}

By Theorem \ref{thm:b-estimate} we get the following upper bound on the
$h$-number.

\begin{corollary}\label{h-bound}
If $b(X)< \infty$ we have 
$$h(X) \leq \dim{\H}^{n-1}(X,\Omega_X^1) + p_g(X).
$$
\end{corollary}
The following corollary was already obtained in \cite{GK4}. In view of its
interest we state it in our new framework.
\begin{corollary}
Let $X$ be a Calabi-Yau variety of dimension $n \geq 1$. Then 
if $h(X) < \infty$ we have
$$
  h(X) \leq \dim {\H}^{n-1}(X, \Omega^{1}_{X}) + 1
$$
In particular, if $X$ is rigid, then $h(X) = 1$ or $\infty$.
\end{corollary}
\begin{proof}
Note that Theorem \ref{thm:b-estimate} implies that $b(X)<\infty$.
Since $\dim {\H}^{n}(X,{\Ocal}_{X}) = 1$, the first inequality follows
from Corollary \ref{h-bound}. If $X$ is
rigid then ${\H}^{1}(X, \Theta_{X}) = 0$ by definition
and by ${\H}^{1}(X, \Theta_{X}) \cong {\H}^{n-1}(X, \Omega_{X}^{1})$,
the conclusion follows from the inequality.
\end{proof}
\end{section}
\begin{section}{Relations between the $a$-number and the $b$-number}
In this section we shall assume that the Hodge-to de Rham spectral
sequence degenerates at the $E_1$-term.

\begin{theorem}
If $b(X)<\infty$ then $a(X)=0$ or $a(X)=1$.
\end{theorem}
\begin{proof}
Consider the commutative diagram
\begin{displaymath}
\begin{xy}
\xymatrix{
{\H}^{n-1}(X,{\Ocal}_X) \ar[d]^{D_1} \ar[dr]^d \\
{\H}^{n-1}(X,B_1) \ar[r]^{\varphi_1} & {\H}^{n-1}(X,\Omega_X^1) \, .\\
}
\end{xy}
\end{displaymath}
Since by our assumption the Hodge-to-de Rham spectral sequence degenerates at
the $E_1$-term the map $d$ is zero. Since by Lemma \ref{b-lemma} the
kernel $\ker \varphi_1$ is contained in the image ${\rm Im}\, D_1$
we see that $\ker \varphi_1={\rm Im} \, D_1$ and we thus have an
injective homomorphism
$$
{\H}^{n-1}(X,B_1)/{\rm Im}\, D_1 \hookrightarrow {\H}^{n-1}(X, \Omega_{X}^1).
$$
If $a(X)\geq 1$ then there exists a non-zero element
$\alpha \in {\H}^{n}(X, {\Ocal}_{X})$ such that $F(\alpha) \in {\rm F}^{1}$.
This means we have an element
$F(\alpha) \in {\rm F}^{1}/{\rm F}^{2} \cong {\H}^{n-1}(X,\Omega_{X}^{1})$.
The exact sequence
$$
0 \rightarrow {\Ocal}_{X} {\buildrel F \over \longrightarrow} {\Ocal}_{X}
\longrightarrow B_{1} \rightarrow 0,
$$
gives rise to the long exact sequence
$$
\to {\H}^{n-1}(X,{\Ocal}_X) {\buildrel D_1 \over \longrightarrow}
{\H}^{n-1}(X,B_1) {\buildrel \delta \over \longrightarrow} {\H}^{n}(X,{\Ocal}_X)
{\buildrel F \over \longrightarrow } {\H}^{n}(X,{\Ocal}_X) \, .
$$
Now $F(\alpha) = 0$ in ${\rm F}^{0}/{\rm F}^{1}
\cong {\H}^{n}(X, {\Ocal}_{X})$
and this tells us that there exists an element
$\beta \in {\H}^{n-1}(X, B_{1})$
such that $\delta (\beta) = \alpha$.  Since $\alpha \neq 0$,
we see $\beta \notin {\rm Im}\, D_{1}$.
Using \v{C}ech comomology, it is easy to see that
$F(\alpha) = \varphi_{1}(\beta)$ in ${\H}^{n-1}(X, \Omega_{X}^{1})$.
Since $F(\alpha) = \varphi_{1}(\beta) \neq 0$, we see  $F(\alpha) \neq 0$
in ${\H}^{n-1}(X, \Omega_{X}^{1})$. Therefore, we have
 $F(\alpha) \notin {\rm F}^{2}$ and it follows that $a(X) = 1$.
\end{proof}

In particular, if the $h$-number is defined then $b(X)<\infty$ implies
$h(X)< \infty$ and then for K3 surfaces and abelian varieties our result
implies results like those in  
\cite{Og}, \cite{GK1}, \cite{GK2} and \cite{GK3}, Proposition 9.4.

\begin{proposition}
For the variety $X$ we have:
$a(X) = 0$ if and only if $b(X)=0$.
\end{proposition}
\begin{proof}
If $a(X)=0$ then in the Hodge filtration we have 
${\rm Im}\, F\cap {\rm F}^1= (0)$
resulting in an isomorphism
$$
F: {\H}^n(X,{\Ocal}_X) \cong {\rm F}^0/{\rm F}^1 \cong {\H}^n(X,{\Ocal}_X)\, .
$$
Therefore the exact sequence 
$$
0 \to {\Ocal}_X {\buildrel F \over \longrightarrow} {\Ocal}_X 
{\buildrel d \over \longrightarrow} 
B_1 \to 0 \eqno(1)
$$ 
implies that ${\H}^n(X,B_1)=(0)$. 
The exact sequence 
$$
0\to B_1 \to B_{i+1}{\buildrel C \over \longrightarrow} B_i\to 0
$$ 
gives inductively ${\H}^n(X,B_i)=(0)$ for all $i>0$. 
This implies that $b(X)=0$.

Conversely, if $b(X)=0$ then in particular ${\H}^n(X,B_1)=(0)$. 
By the exact sequence (1) we see that 
$F: {\H}^n(X,{\Ocal}_X)\to {\H}^n(X,{\Ocal}_X)$ is
surjective and thus an isomorphism. 
Since ${\H}^{n}(X, {\Ocal}_{X})\cong {\rm F}^{0}/{\rm F}^{1}$, 
we conclude that $({\rm Im}\, F )\cap {\rm F}^{1} = (0)$ and so $a(X)=0$.
\end{proof}
\end{section}
\begin{section}{Fermat surfaces}
As an example we now calculate the $a$-number of 
Fermat surfaces and deduce consequences for the $h$-number.
Recall that a non-singular complete curve $C$ is said to be  
{\sl ordinary} if its Jacobian variety is an ordinary abelian variety.
Note that a curve $C$ is ordinary if and only if Frobenius induces a bijective
map from ${\rm H}^1(C,O_C)$ to its image. Hence $a(C)=0$ if 
$C$ is ordinary and $a(C)=1$ otherwise.

\begin{proposition}
Let $C_{1}$ and $C_{2}$ be non-singular complete algebraic curves
defined over $k$. Then the $a$-number of $X=C_{1}\times C_{2}$ satisfies
$$
a(C_1\times C_2)= \# \{i: 1\leq i \leq 2, \hbox{\rm $C_i$ is not ordinary}\}
$$
\end{proposition}
\proof{ If both $C_{1}$ and $C_{2}$ are ordinary, the Frobenius map
acts bijectively  on
${\rm H}^{2}(X, { O}_{X}) = 
{\rm H}^{1}(C_{1}, {O}_{C_{1}}) \otimes {\rm H}^{1}(C_{2}, { O}_{C_{2}})$
and $a(X)=0$. Consider for $i=1,2$
the Hodge filtration of de Rham cohomology
$$
{\rm H}_{DR}^{1}(C_{i})=F_{(i)}^{0} \supset F_{(i)}^{1} \supset (0) \,  .
$$
The Hodge filtration
$$
{\rm H}_{DR}^{2}(X) =F^{0} \supset F^{1} \supset F^{2} \supset (0)
$$
of de Rham cohomology ${\rm H}^{2}_{DR}(X)$  has $F^2$ given by
$$
F^{2} = F_{(1)}^{1} \otimes F_{(2)}^{1}\, .
$$
We consider the Frobenius map
$$
{\rm H}^{2}(X, { O}_{X}) \cong {\rm H}^{1}(C_{1}, {O}_{C_{1}}) 
\otimes {\rm H}^{1}(C_{2}, { O}_{C_{2}}) 
\stackrel{F\otimes F}{\longrightarrow} {\rm H}^{2}_{DR}(X)
$$
If both $C_{i}$ are non-ordinary, then there exists an element $\alpha_{i}$ in 
${\rm H}^{1}(C_{i}, {O}_{C_{i}})$ such that $F(\alpha_{i}) = 0$ on
${\rm H}^{1}(C_{i}, {O}_{C_{i}})$.
This means that
$$
(F\otimes F)(\alpha_{1} \otimes \alpha_{2}) \in F^{2} 
= F_{(1)}^{1} \otimes F_{(2)}^{1}
$$
and $a(X)=2$.
If exactly one is ordinary then the image of ${\rm H}^2_{\rm dR}(X,O_X)$
lies in $F^1 \cap {\rm H}^1_{\rm dR}(C_1) \otimes {\rm H}^1_{\rm dR}(C_2)$
and does not lie in $F_{(1)}^1\otimes F_{(2)}^1$.
We thus see that the $a$-number equals the number of non-ordinary factors.
}

\begin{remark}
If both $C_{1}$ and $C_{2}$ are non-ordinary, then $a(X) = 2$, and hence
in this case the $h$-number of $C_{1}\times C_{2}$ is equal to $\infty$.
\end{remark}

Let $X_{m}$ be the Fermat surface defined in ${\bf P}^3$ by the homogeneous
equation
$$
      z_{0}^{m} + z_{1}^{m} + z_{2}^{m}+ z_{3}^{m} = 0\, .
$$
We assume $m\geq 4$ and that $m$ is prime to the characteristic $p$.

We shall calculate the $a$-number of $X_m$ by  using 
the inductive structure of Fermat varieties
as employed in  \cite{Sa} and \cite{SK}.
For this we define the Fermat curve $C_{m}$  by the equation
$$
         x_{0}^{m} + x_{1}^{m} + x_{2}^{m} = 0\, .
$$
By  \cite{Sa} and \cite{SK} we have a rational map
$$
    \varphi :  C_{m} \times C_{m}  \longrightarrow X_{m}
$$
defined by 
$$
((x_0,x_1,x_2),(y_0,y_1,y_2))\mapsto (x_0y_2,x_1y_2,\epsilon x_2y_0,\epsilon x_2y_1)\, ,
$$
where $\epsilon$ is a fixed $2m$-th root of unity with $\epsilon^{m} = -1$.

The rational map $\varphi$ is not defined at the $m^2$ points where both
$z_2$ and $y_2$ vanish.
Let $Z_{m}$ be the surface obtained by blowing up 
$C_{m} \times C_{m}$ at these $m^{2}$ points. 

An element $\zeta$  of the group $\mu_{m}$ of $m$-th roots of unity 
acts on $C_{m} \times C_{m}$ via
$$
((x_{0}, x_{1}, x_{2}), (y_{0}, y_{1}, y_{2})) \mapsto 
((x_{0}, x_{1}, \zeta x_{2}), (y_{0}, y_{1}, \zeta y_{2}))\, .
$$
We set $G = \mu_{m}$.
The fixed point set of this action is equal to the locus of indeterminacy of 
$\varphi$  and this action naturally 
extends after blowing up to $Z_{m}$. We have the following diagram
\begin{displaymath}
\begin{xy}
\xymatrix{
Z_m \ar[r]^{\psi} \ar[d]^{\tilde{\varphi}}& C_m\times C_m \ar[d]^{\varphi} \\
Z_m/G \ar[r]^{\phi} & X_m
}
\end{xy}
\end{displaymath}
Here, the quotient surface $Z_{m}/G$ is nonsingular and
$\phi$ contracts $2m$ non-singular rational curves. For the details of
this construction we refer to \cite{Sa} or \cite{SK}. 
We derive a diagram in cohomology

\begin{displaymath}
\begin{xy}
\xymatrix{
{\rm H}^1(C_m,O_{C_m})\otimes {\rm H}^1(C_m,O_{C_m}) \ar[r] & {\rm H}^2(Z_m,O_{Z_m})  \\
{\rm H}^2(X_m,O_{X_m}) \ar[r] & \ar[u]^{\tilde{\varphi}^*}{\rm H}^2(Z_m/G,O_{Z_m/G}) 
}
\end{xy}
\end{displaymath}
where the horizontal arrows are isomorphisms and the vertical map identifies
the cohomology ${\rm H}^2(Z_m/G,O_{Z_m/G})$ 
with the invariants ${\rm H}^2(Z_m,O_{Z_m})^G$.
We conclude that ${\rm H}^2(X_m,O_{X_m})$ equals the invariant subspace 
$({\rm H}^{1}(C_{m}, O_{C_{m}})\otimes {\rm H}^{1}(C_{m}, O_{C_{m}}))^{G}$.

In order to calculate the action of $G = \mu_{m}$ on 
the cohomology group ${\rm H}^{1}(C_{m}, O_{C_{m}})$.
we consider the open  covering 
$ U_{i} = \{ (x_{0}, x_{1}, x_{2}) \in C_{m} \mid x_{i} \neq 0\}$
with $i \in \{0,1\}$.
The functions $t_{1} = x_{1}/x_{0}$ and $t_{2} = x_{2}/x_{0}$ 
define affine coordinates
on the curve $U_0$ given by 
$$
  1 + t_{1}^{m} + t_{2}^{m} = 0 \, .
$$
We represent elements of $H^{1}(C_{m}, O_{C_{m}})$ as \v{C}ech cocycles
with respect to the affine open covering $\{U_{0}, U_{1}\}$.
They are represented by regular functions on $U_{0}\cap U_{1}$. 
We set
$$
    \alpha_{a, b} = t_{2}^{b}/t_{1}^{a}\, .
$$
\begin{lemma}\label{boundary}
If $a$ is not positive or $a\geq b$  then $\alpha_{a, b}$ is cohomologous to zero.
\end{lemma}
\begin{proof}{If $a$ is non-positive, then $\alpha_{a, b}$ is 
regular on $U_{0}$ and  the co-cycle 
$\omega = (- \alpha_{a, b},  0) \in 
\Gamma(U_{0}, 0_{C_{m}}) \oplus \Gamma(U_{1}, 0_{C_{m}})$ gives
 $\delta (\omega) = \alpha_{a, b}$.
Similarly, $a \geq b$, then $\alpha_{a, b}$ is regular on $U_{1}$;
then the co-cycle $\omega = (0, \alpha_{a, b}) \in 
\Gamma(U_{0}, 0_{C_{m}}) \oplus \Gamma(U_{1}, 0_{C_{m}})$ gives $\delta (\omega) = \alpha_{a, b}$.
}
\end{proof}
We let
$$
\Xi = \{(a, b) \in {\ZZ}\times {\ZZ}:  1\leq a< b \leq m-1\} \, .
$$

\begin{proposition} A basis of ${\rm H}^1(C_m,O_{C_m})$ is given by the set of 
cocycles 
$$
\{ \alpha_{a,  b} = t_{2}^{b}/t_{1}^{a}:   (a, b) \in \Xi \} \, .
$$
\end{proposition}
\begin{proof}{
Note that the cardinality of this set equals the dimension of 
${\rm H}^1(C_m,O_{C_m})$. 
So it suffices to show that these co-cycles generate this cohomology group. 
A regular function $f$ on $U_{0}\cap U_{1}$ can be considered as 
a rational function on $U_{0}$ with poles only at the points 
given by $t_{1} = 0$. So $f$ is a linear combination of the 
functions $\alpha_{c, d}$ with $c$ an integer and $d$ a non-negative 
integer. By Lemma \ref{boundary}
we can assume $c\geq 1$. If $d \geq m$, then
we have
$$
\alpha_{c, d} = t_{2}^{d -m}t_{2}^{m}/t_{1}^{c} 
                       = t_{2}^{d-m}(-1 -t_{1}^{m})/t_{1}^{c}
                      = -\alpha_{c, d-m} - \alpha_{c -m, d -m}
$$
showing that we can assume $d \leq m -1$. 
}
\end{proof}
The action of $\zeta \in G =\mu_{m}$ on ${\rm H}^{1}(C_{m}, O_{C_{m}})$ is 
obviously given by
$$
  \alpha_{a, b} \mapsto \zeta^{b} \alpha_{a, b}.
$$  
We write our prime as
$$
  p = d + nm \quad\hbox{\rm with $1\leq d \leq m-1$}.
$$
For a integer $\ell$, if $\ell \equiv k ~(\mbox{mod}~m)$ with $0\leq k \leq m-1$,
we set $\bar{\ell} = k$. Using this notation, we have $d = \bar{p}$.
Then, by Lemma \ref{boundary}
the Frobenius action on ${\rm H}^{1}(C_{m}, O_{C_{m}})$ is given by
$$
\begin{array}{cl}
 F^{*}\alpha_{a, b} & = t_{2}^{bp}/t_{1}^{ap} 
=t_{2}^{bd}(t_{2}^{m})^{nb}/t_{1}^{ad}t_{1}^{mna}\\
  & = t_{2}^{bd}(- t_{1}^{m} - 1)^{nb}/t_{1}^{ad}t_{1}^{man}\\
  & = (-1)^{nb}t_{2}^{bd}
(\sum_{i = 0}^{nb}{nb \choose i} t_{1}^{(i-an)m})/t_{1}^{ad}\\
\end{array}
$$
and this last expression is cohomologous to
$$ 
(-1)^{nb}t_{2}^{\overline{bd}}/t_{1}^{\overline{ad}} \, .
$$
We know by Lemma \ref{boundary}
that
$(-1)^{nb}t_{2}^{\overline{bd}}/t_{1}^{\overline{ad}}$
is zero in $H^{1}(C_{m}, O_{C_{m}})$ if and only if $\overline{bd} \leq \overline{ad}$.
Using this result (or by Koblitz \cite{Ko}), we get the following lemma.

\begin{lemma}\label{isoorzero}
The Frobenius map induces a $p$-th linear isomorphism on 
${\rm H}^{1}(C_{m}, O_{C_{m}})$ if and only if $p \equiv 1 ~({\mbox{mod}}~m)$.
The Frobenius map induces the zero map on ${\rm H}^{1}(C_{m}, O_{C_{m}})$
 if and only if $p \equiv - 1 ~({\mbox{mod}}~m)$. 
\end{lemma}
\begin{remark}
By this lemma and Ibukiyama-Katsura-Oort \cite{IKO}, 
for $p \equiv -1~(\mbox{mod}~m)$
the Jacobian variety of $C_{m}$ is isomorphic to a product of supersingular
elliptic curves.
\end{remark}

\begin{theorem} Let $m$ be an integer $\geq 4$ and prime to $p$. Then the $a$-number of the Fermat surface $X_m$ is given by
$$
a(X_m)=\begin{cases}
0 & p\equiv 1 (\bmod \, m) \\
1 & p \equiv 2 \quad{or}\quad  2^{-1} \quad \hbox{in $({\ZZ}/m{\ZZ})^*$} \\
2 & otherwise \\
\end{cases}
$$
\end{theorem}
\begin{proof}
Write $C=C_m$ and $X=X_m$ and let
$$
{\rm H}^{1}_{\rm dR}(C)= F_{C}^{0} \supset F_{C}^{1} \supset \{0\}
$$
be the Hodge filtration of ${\rm H}^{1}_{\rm dR}(C)$.
Since the image of ${\rm H}^2(X,O_X)$ lands in 
the $G$-invariant part of
${\rm H}_{\rm dR}^1(C)
\otimes {\rm H}^1_{\rm dR}(C)$ in ${\rm H}^2_{\rm dR}(X)$ 
we consider the Hodge filtration on 
$({\rm H}_{\rm dR}^1(C)\otimes {\rm H}^1_{\rm dR}(C))^G$ 
$$
({\rm H}_{\rm dR}^1(C)
\otimes {\rm H}^1_{\rm dR}(C))^G=F^{0}\supset F^{1}\supset F^{2}\supset \{0\}
$$
Then we have
$$
F^{1} =  \{F_{C}^{1}\otimes F^{0} + F^{0}\otimes F_{C}^{1}\}^{G}
\quad \hbox{\rm and} \quad
F^{2} = \{F_{C}^{1}\otimes F_{C}^{1}\}^{G}\, .
$$
Thus by Lemma \ref{isoorzero} we see that
$a(X_{m}) = 0$ if and only if $p \equiv 1 (\bmod \, m)$. Moreover,
if $p\equiv -1 (\bmod \, m)$ then $a(X_{m}) = 2$.
From here on we assume $p\not\equiv \pm 1 (\bmod \, m)$, 
i.e.,  $2 \leq d \leq m - 2$. In this case either $a(X_m)=1$ or $a(X_m)=2$.
We look for a $G$-invariant element 
$$
\gamma=\gamma_{a,a',b} = \alpha_{a,b}\otimes \alpha_{a',m-b}
$$ 
with $(a, b)$ and $(a', m-b)$ in $\Xi$ such that
$F^{*}\alpha_{a,b} = 0$, $F^{*}\alpha_{a',m-b} = 0$ in 
${\rm H}^{1}(C_{m}, O_{C_{m}})$
If such a $\gamma$ exists, 
then the Frobenius image of $\gamma$ in ${\rm H}_{dR}^{2}(X_{m})$ is contained in
$F^{2}$ and so we have $a(X_{m}) = 2$. If such an element doesn't exist, 
then we see $a(X_{m}) = 1$.
Hence we are reduced to the following combinatorial
problem: $a(X_m)=2$ if and only if the set $Y(m,d)$
$$
\{ (a,a',b): 1\leq a, a', b \leq m-1;  a<b, a'<m-b, 
\overline{da}\geq \overline{db}, \overline{da'}\geq \overline{d(m-b}\} 
$$
is not empty.

We first deal with the case $d=2$. This implies that $m$ is odd.
 If $a<b$ and $\overline{2b}\leq 
\overline{2a}$ then we have $(m+1)/2 \leq b \leq (m-1-a)/2$. Similarly, if
$a'<m-b$ and $\overline{2(m-b)}\leq \overline{2a'}$ then 
$(m+1)/2 \leq m-b \leq (m-1-a')/2$; but then $m=b+(m-b)\geq m+1$, a contradiction. Thus $a(X_m)=2$ in this case.

We set $\ell=[m/d]$. We shall assume that $d>2$. If $\ell>3$ or $\ell=2$ 
then the element $(\ell,[(d-2)m/d],\ell+2)$ is in $Y(m,d)$ or (if $\ell=2$
and $r<d/2$) the element $(2,[(d-2)m/d],3)$ is in $Y(m,d)$. 

Suppose now that $\ell=1$ and $m\neq 2d-1$. 
Since ${\rm gcd}(d,m)=1$ we have integers $x,y$ with $y>0$ 
such that  $xd+ym=1$. Put $b=\overline{1-x}$. We then can write 
$bd=d-1+ym$. Then the element $(b-1,[(d-y-1)m/d],b)$ is in $Y(m,d)$ as 
the reader may check.  Finally, if $p\equiv 2^{-1} (\bmod \, m)$, i.e., 
$m=2d-1$, consider the interval $[i\ell,(i+1)\ell]$. It contains two integers, say $b,b+1$. Suppose we have an $a$ such that 
$1\leq a < b$ with $\overline{da}\geq \overline{db}$. Then $m-b$ is the larger
integer in the interval
$[j\ell,(j+1)\ell]$. Therefore if $a'<m-b$ we will have
$\overline{da'}<\overline{d(m-b)}$ and the set $Y(m,d)$ is empty.
\end{proof}

The reader may check that $Y(m,d)$ and $Y(m,d')$ for $dd'=1 (\bmod \, m)$ 
have the same cardinality; this shortens the proof.

\begin{corollary} For the Fermat surface $X_m$ in characteristic $p$
(with $m\geq 4$, ${\rm gcd}(m,p)=1$) the height $h(X_m)$
equals $1$ if and only if $p\equiv  1 \, (\bmod \, m)$. Furthermore, 
If $p\not\equiv 1, 2, 2^{-1} \, (\bmod \, m)$  we have  $h(X_m) = \infty$.
\end{corollary}
\end{section}
\noindent
{\bf Acknowledgement} The second author was supported in part by JSPS Grant-in-Aid (S), No 19104001.

\end{document}